\documentclass[11pt,a4paper]{article} 

\usepackage{amsfonts, amsmath, wasysym}
\usepackage{stmaryrd} 

\usepackage{tikz}
\usetikzlibrary{shapes, positioning, patterns, arrows.meta}
\usepackage{caption}

\usepackage{amssymb,amsthm,
paralist
}

\usepackage{
latexsym,
}

\usepackage{url}

\definecolor{darkgreen}{rgb}{0,0.5,0}
\definecolor{darkred}{rgb}{0.7,0,0}
\usepackage[colorlinks, 
citecolor=darkgreen, linkcolor=darkred
]{hyperref}

\usepackage[a4paper, margin=1in]{geometry}

\theoremstyle{plain}

\numberwithin{equation}{section}

\newcommand{\h}{\ensuremath{{\mathcal H}}}
\newcommand{\s}{\ensuremath{{\mathcal S}}}
\newcommand{\w}{\ensuremath{{\mathcal W}}}

\newcommand{\ca}{\ensuremath{{\mathcal A}}}

\newcommand{\cc}{\ensuremath{{\mathcal C}}}

\newcommand{\ci}{\ensuremath{{\mathcal I}}}

\newcommand{\Hv}{{\bf H}}

\newcommand{\Ga}{\Gamma}

\newcommand{\Om}{\Omega}
\newcommand{\ka}{\kappa}

\newcommand{\si}{\sigma}

\newcommand{\Si}{\Sigma}
\renewcommand{\th}{\theta}
\newcommand{\Th}{\Theta}

\newcommand{\ep}{\varepsilon}

\newcommand{\R}{\ensuremath{{\mathbb R}}}
\newcommand{\N}{\ensuremath{{\mathbb N}}}

\newcommand{\downto}{\downarrow}

\newcommand{\embed}{\hookrightarrow}

\newcommand{\lap}{\Delta}

\newcommand{\grad}{\nabla}

\newcommand{\intersect}{\cap}

\DeclareMathOperator{\Interior}{Interior}

\DeclareMathOperator{\area}{Area}
\DeclareMathOperator{\diam}{diam}
\DeclareMathOperator{\Length}{Length}

\newcommand{\beq}{\begin{equation}}
\newcommand{\beql}[1]{\begin{equation}\label{#1}}
\newcommand{\eeq}{\end{equation}}
\newcommand{\beqa}{\begin{equation}\begin{aligned}}
\newcommand{\eeqa}{\end{aligned}\end{equation}}
\newcommand{\brmk}{\begin{rmk}}
\newcommand{\ermk}{\end{rmk}}
\newcommand{\partref}[1]{\hbox{(\csname @roman\endcsname{\ref{#1}})}}
\newcommand{\half}{\frac{1}{2}}

\usepackage{soul}

 \newtheorem{thm}{Theorem}[section]
\newtheorem{cor}[thm]{Corollary}
\newtheorem{lem}[thm]{Lemma}

\newtheorem{defn}[thm]{Definition}

\theoremstyle{remark}
\newtheorem{rmk}[thm]{Remark}

\usepackage{varwidth}

\title{\Large \sc sharp diameter estimates via integral geometry}
\author{Peter M. Topping}
\date{25 October 2026}

\begin{document}

%

\parskip 8pt
\parindent 0pt

\maketitle

\begin{abstract}
We prove that the diameter of a closed embedded surface in $\R^3$ is 
always less than the integral of the magnitude of the mean curvature, divided by $\pi$.
We also prove a sharp variant for surfaces with boundary, and apply it to address a conjecture of Miura.
\end{abstract}

\section{Introduction}
\label{intro}

\subsection{Background}

In a publication from 1998 \cite{MCF_iso_ineq} we conjectured that the diameter of a
closed surface $M$ immersed in $\R^3$  is controlled by 
\beq
\label{conjectured_est}
\diam(M)<\frac{1}{\pi}\int_M |H|dvol,
\eeq
where $diam(M)=\max_{x,y\in M}|x-y|$ is the extrinsic diameter, and $H$ is the mean curvature, normalised as the average of the principal curvatures rather than the sum, so that the unit sphere has $|H|\equiv 1$. 
Here, and below, $M$ is implicitly assumed to be connected.
The constant $\frac{1}{\pi}$ could never be lowered, as can be seen by considering a long finite cylinder capped by hemispheres.
In \cite{topwent} we had already proved, using compensation properties of the Jacobian determinant, that in the special case that $M$ has \emph{constant} mean curvature, we have an estimate 
$$\diam(M)\leq\frac{\area(M)|H|}{2\pi}=\frac{1}{2\pi}\int_M |H|dvol,$$
which is twice as strong as \eqref{conjectured_est}. Moreover, 
in \cite{MCF_iso_ineq} we observed that optimising constants in an estimate of Leon Simon \cite{leon_CAG} gave us
$$\diam(M)<\frac{2}{\pi}\left[\area(M)\int_M|H|^2dvol\right]^\half,$$
which is half as strong as what would be given by weakening the conjecture  
\eqref{conjectured_est} with an application of the Cauchy-Schwarz inequality.

In 2008, we proved that such an inequality holds even in higher dimensions and co-dimensions, and even controlling the \emph{intrinsic} diameter, provided one is willing to trade the precise constant $\frac{1}{\pi}$ in \eqref{conjectured_est} for a dimensional constant.
\begin{thm}[From \cite{topping_diam_cmh}]
\label{cmh_thm}
For $m\in \N$, suppose that $M$ is an $m$-dimensional closed 
manifold smoothly immersed in $\R^n$. 
Then its intrinsic diameter $\diam_{int}$
and its mean curvature vector $\Hv$ are related by
\beq
\diam_{int}\leq C(m)\int_M |\Hv|^{m-1} dvol.
\eeq
\end{thm}
The method of proof was inspired by our diameter estimates for Ricci flow \cite{diamrf}.
The key lemma, specialised to the simplest case of surfaces in $\R^3$, was the following.

\begin{lem}
Suppose that $M$ is a closed surface immersed in $\R^3$,
and $x\in M$. Then 
\beq
\label{nice2Dest}
\sup_{r>0} \frac{1}{r}\int_{B(x,r)}|H|dvol >\frac{\pi}{8},
\eeq
where $B(x,r)$ is the intrinsic ball centred at $x$ of radius $r$.
\end{lem}
For any $M$, the 
quantity $\frac{1}{r}\int_{B(x,r)}|H|dvol$
is made arbitrarily
small by sending $r>0$ to either zero or infinity; the lemma says
that there is always some intermediate $r>0$ where it has a universal positive
lower bound.
The proof of Theorem \ref{cmh_thm} for surfaces in $\R^3$ involved taking a minimising geodesic of maximal length, and  covering a fixed proportion of it with a disjoint collection of balls $B(x,r)$ that achieved the maximum in \eqref{nice2Dest}. The diameter could be related to the sum of the radii of all these balls, which in turn could be related to the integral of $|H|$ over the union of these disjoint balls.

Since 2008, Theorem \ref{cmh_thm} and its generalisations have been applied to diverse topics
such as min-max theory for the construction of prescribed mean curvature surfaces \cite{mazurowski20}, almost-CMC hypersurfaces \cite{ciraolo_maggi_21}, mean curvature flow \cite{liu_etal27, lin_sesum29}, 
volume-preserving mean curvature flow \cite{li22, julin_etal24, huang32, julin36}, 
surface area preserving mean curvature flow \cite{huang28},
surface diffusion flows \cite{mccoy_etal23}, 
K\"ahler geometry \cite{tosatti25, zhang37, bednarek42}, isoperimetric problems \cite{julin30},
almost-umbilic hypersurfaces \cite{roth33}  and mathematical relativity \cite{huang38},  amongst others \cite{bates39, speciel35}.

One case in which the conjecture in \eqref{conjectured_est}
has already been established is the case of simply connected axisymmetric  surfaces.
This was initially dealt with by Helmers in \cite[Lemma 2.4]{helmers}, who applied the result to give a compactness 
result in the study of two-phase lipid bilayer membranes.
Miura \cite{miura_selecta} gave an 
alternative approach to the axisymmetric case and gained substantial extra information concerning 
error terms and the behaviour of a sequence of axisymmetric surfaces as the inequality is saturated.
Li-Wan-Zhang \cite{li_wan_zhang} developed these ideas in higher dimensions.

Theorem \ref{cmh_thm} has also been generalised in various directions. We will shortly discuss the case of surfaces with boundary. We highlight also the efforts to generalise the result to spaces of weaker regularity (see particularly Menne--Scharrer \cite{MS1}
and Scharrer \cite{scharrer_spontaneous})  to ambient Riemannian manifolds (see for example Wu--Zheng \cite{wu_zheng},  
Chambers--Marx-Kuo \cite{C_MK} and Impera--Rimoldi--Veronelli \cite{IRV}), 
to weighted ambient manifolds (de Medeiros \cite{Medeiros}),
and to conformally flat ambient spaces (Flaim-Scharrer \cite{flaim_scharrer}).


\subsection{New results}

The purpose of this paper is to prove the original conjecture for closed embedded surfaces in $\R^3$. In fact, we prove a  stronger result that requires the elementary  concept of \emph{width} of $M$.
\begin{defn}
Given a unit vector $e\in \R^3$, the \emph{width} of a closed embedded surface $M$ in $\R^3$, or indeed of any closed bounded subset $M$ of $\R^3$, is defined to be 
$$\w_M(e):=\max_{x\in M}\langle x,e\rangle - \min_{x\in M}\langle x,e\rangle.$$
\end{defn}
Thus, if we clamp  $M$ in a vice with two jaws perpendicular to $e$, then the width $\w_M(e)$ is the distance between these two jaws.

\begin{thm}[Main theorem]
\label{main_thm}
Suppose that $M$ is a closed surface embedded in $\R^3$.
Then 
\beq
\label{gen_conc}
\int_M H_+ dvol\geq \half \int_{S^2}\w_M(e) d\si(e)>\pi \diam(M).
\eeq
In particular,
\beq
\diam(M)<\frac{1}{\pi}\int_M |H|dvol.
\eeq
\end{thm}
Here we are adopting the most standard sign convention that if $\nu$ is an outward normal vector to $M$ then the mean curvature vector is $-H\nu$, that is, the mean curvature of the unit sphere is $+1$. We are also writing $H_+$ to mean $\max\{H,0\}$
and writing $d\si$ for the standard volume element on the unit sphere $S^2\embed \R^3$.

\begin{rmk}
\label{C11regthm}
The statement of Theorem \ref{main_thm} makes sense even when $M$ is of only $C^{1,1}$ regularity because this regularity is enough to be able to view the mean curvature as an $L^\infty$ function on $M$. It is this more general result that we prove in this paper.
\end{rmk}

\begin{rmk}
\label{proof_reduction}
The strict inequality of \eqref{gen_conc} is elementary: Given a closed bounded connected set $M\subset \R^3$, if $x,y\in M$ are picked so that $\diam(M)=|x-y|$, then $\w_M(e)\geq |\langle y-x,e\rangle|$ by definition, with 
equality for all $e\in S^2$ only if $M$ is the line segment connecting $x$ and $y$.
Therefore 
$$\int_{S^2}\w_M(e) d\si(e)\geq \int_{S^2}|\langle y-x,e\rangle| d\si(e)
=|x-y|\int_{-1}^1 |h|\,2\pi dh=2\pi \diam(M),$$
by Archimedes' cylindrical representation of the area element of a 2-sphere,
with equality only if $M$ is the line segment connecting $x$ and $y$, which is not a surface.
\end{rmk}

\begin{rmk}
\label{mink_rmk}
As we discuss later, the first inequality of \eqref{gen_conc} becomes equality in the case of convex surfaces, and in this special case dates back at least to Minkowski in 1903 (see equations (29) and (30) on page 458 of 
\cite{minkowski}). The relevance of this equality to diameter estimates was already identified by Miura \cite{miura_selecta}.
\end{rmk}

We now turn to the topic of the analogous diameter estimates on surfaces with boundary. This general topic was pioneered by Menne and Scharrer with more flexibility on dimensions and regularity, but without sharp constants, see e.g. \cite[Corollary 2]{MS1}, and by Miura, e.g. \cite[Theorem 1.1]{miura_sns}.
Earlier work with a restrictive convexity assumption is due to 
Paeng \cite{paeng7}.
A clever idea of Miura \cite{miura_sns} is that one can prove diameter estimates for 
surfaces with boundary by applying diameter estimates for closed surfaces to the 
$\ep$-distance set of the surface with boundary.
Combining this idea with our new diameter estimate will give the following.
\begin{thm}[The case with boundary]
\label{bdry_thm}
Suppose that $M$ is a smooth compact surface with non-empty boundary, embedded in $\R^3$.
Then 
\beq
\label{bdry_est}
\int_M |H| dvol + \frac{\pi}{2}\Length(\partial M)
\geq \half \int_{S^2}\w_M(e) d\si(e)>\pi \diam(M).
\eeq
In particular,
\beq
\label{bdry_conseq}
\diam(M)<\frac{1}{\pi}\int_M |H|dvol + \half \Length(\partial M).
\eeq
\end{thm}
Both coefficients on the right-hand side of \eqref{bdry_conseq} are sharp: The first is seen to be sharp by considering a long cylinder capped at one end. The second is seen to be sharp by considering a long thin rectangle in the plane with its corners smoothed.

It is notable that the  reason the application of Miura's method gives us a sharp
estimate, rather than one with the coefficient of the mean curvature term degraded by a factor of two, as in the past, will be because our main Theorem \ref{main_thm} involved the integral of the \emph{positive part} of the mean curvature, rather than the absolute value of the mean curvature.

By applying Theorem \ref{bdry_thm} to minimal surfaces with boundary, we resolve a conjecture of Miura \cite{miura_sns} for minimal surfaces with boundary embedded 
in $\R^3$.
\begin{cor}
\label{min_surf_cor}
Suppose that $M$ is a smooth compact minimal surface with  boundary, embedded in $\R^3$.
Then 
$$\Length(\partial M)> 2\diam(\partial M)=2\diam(M).$$
\end{cor}
As throughout, we are assuming our manifold $M$ is connected.
By the convex hull property of minimal surfaces, we know that $\diam(M)=\diam(\partial M)$, so this is indeed a direct corollary of Theorem \ref{bdry_thm}. It is also sharp, by the example from before of a long thin rectangle in the plane with its corners smoothed.

The corollary is trivial in the case that the boundary $\partial M$ is connected, whatever the topology of $M$.
A natural restricted case of Corollary \ref{min_surf_cor} in which a stronger estimate was already effectively known is the case that the minimal surface is a free-boundary minimal surface within the unit ball in $\R^3$. Fraser and Schoen \cite{FS} proved in this case that $\Length(\partial M)\geq 2\pi$, and  clearly have 
$\diam(M)\leq 2$ when the surface lies within the unit ball, so 
$$\Length(\partial M)\geq \pi\diam(M)$$ 
in this special case.

Following Miura \cite{miura_sns}, we also note the following immediate consequence of 
Corollary \ref{min_surf_cor}: if we are given a smooth contour $\Ga\subset \R^3$ consisting of
a finite collection of pairwise disjoint simple closed curves such that 
$\Length(\Ga)\leq 2\diam(\Ga)$, then $\Ga$ cannot be spanned by a compact connected 
embedded minimal surface.

\section{Proofs}

Given Remark \ref{proof_reduction}, the main Theorem \ref{main_thm} 
requires us to prove the first inequality of \eqref{gen_conc}.
As one might guess from the nature of the inequality, there is an integral geometric identity  within the proof. We start by giving that identity, even though it is currently disconnected from both the hypotheses and the conclusion of Theorem \ref{main_thm}.
For $e\in S^2\embed \R^3$ and $t\in\R$, we denote the plane orthogonal to $e$ at `height' $t$ by
$$P_{e,t}:=\{x\in\R^3\ :\ \langle x,e\rangle =t\}.$$
\begin{lem}[Crofton formula for total mean curvature]
\label{crofton_lem}
Suppose $K\subset\R^3$ is the closure of an open bounded set whose boundary is 
$C^{1,1}$. Then
\beql{crofton_identity}
\int_{\partial K} H dvol = \half\int_{S^2}\int_{\R} \chi(K_{e,t})dt d\si(e),
\eeq
where 
$K_{e,t}:=K\intersect P_{e,t}$ is a slice of $K$, and $\chi$ represents the Euler characteristic.
\end{lem}

\begin{rmk}
\label{IFT_rmk}
Because $\partial K$ is assumed to be $C^{1,1}$, the Gauss map $\nu:\partial K \to\ S^2$, chosen to be the outward unit normal to $K$, is Lipschitz. Therefore, the area formula applies to give 
$$\int_{S^2} \h^0(\nu^{-1}(e)) d\si(e) = \int_{\partial K}|\det(\grad\nu)|dvol<\infty.$$
Thus, for almost every $e\in S^2$, both $\nu^{-1}(e)$ and $\nu^{-1}(-e)$ are finite.
The union of these two sets is the set of critical points of the 
height function $h_e:\partial K\to\R$ defined by $h_e(x)=\langle x,e\rangle$,
and thus $h_e$ has a finite set of critical values. 
In particular, for each $t\in \R$ off this finite set of critical values, the level set $h_e^{-1}(t)$ is the  
union of a finite number of pairwise disjoint  $C^{1,1}$ embedded closed curves
in $P_{e,t}$, which together form the boundary of $K_{e,t}$. This allows us to divide up the curves into `outer' and `inner' curves in a natural way. A boundary curve would be `inner' if that boundary component would disappear if we were to add  the disc it bounds in $P_{e,t}$ to $K_{e,t}$; otherwise it is `outer'.
For such $(e,t)$, in particular for 
$(e,t)$ off a set 
of measure zero in the domain of integration $S^2\times \R$, the 
integrand $\chi(K_{e,t})$ is the number of outer curves minus the number of inner curves, as just described. See Figure \ref{Ket_fig}.
\end{rmk}

The Crofton formula of Lemma \ref{crofton_lem} could be extracted as a special case of results in the literature. The result, albeit with stronger regularity hypotheses, can be found in the work of Santaló \cite[(4.28), page 37]{santalo} from 1952, 
while less regular surfaces are handled by an instance of a result of 
Federer \cite[Theorem 6.13]{federer} from 1959. We give a self-contained proof for the convenience of the reader.

\begin{proof}[Proof of Lemma \ref{crofton_lem}]
We adopt the notation and observations of Remark \ref{IFT_rmk}.
For every $(e,t)\in S^2\times \R$, and at every regular point of $h_e$ that we can find on $\partial K_{e,t}$, we write $n$ for the outward unit normal to $\partial K_{e,t}$ within $P_{e,t}$, and define a unit tangent vector to $\partial K_{e,t}$ within $P_{e,t}$ by $T_e:=e\times n = \frac{e\times \nu}{|e\times \nu|}$. 
See Figure \ref{Ket_fig}.

\begin{figure}
\captionsetup{width=0.85\textwidth}
\centering
\begin{tikzpicture}[
  x=1cm,y=1cm,
  font=\normalsize,
  line cap=round,
  line join=round,
  boundary/.style={draw=black,line width=0.8pt},
  vector/.style={draw=black,line width=0.8pt,
    -{Stealth[length=2.2mm,width=1.7mm]}},
  leader/.style={draw=black!65,line width=0.45pt,
    -{Stealth[open,length=1.8mm,width=1.3mm]}},
  vector label/.style={inner sep=1.5pt}
]
\definecolor{sliceBlue}{RGB}{220,235,247}

\def\largeouter{
  (9.20,0.15)
  .. controls (9.20,1.55) and (8.52,3.00) .. (7.10,3.00)
  .. controls (5.90,3.00) and (3.45,2.75) .. (3.45,2.15)
  .. controls (3.45,1.55) and (4.95,1.05) .. (4.95,0.15)
  .. controls (4.95,-0.70) and (3.70,-1.03) .. (3.70,-1.80)
  .. controls (3.70,-2.45) and (4.45,-2.55) .. (5.40,-2.55)
  .. controls (7.55,-2.55) and (9.20,-1.50) .. (9.20,0.15)
  -- cycle
}

\def\innercurve{
  (7.00,-1.35)
  .. controls (6.50,-1.35) and (6.50,-0.75) .. (6.55,-0.25)
  .. controls (6.60,0.25) and (6.22,0.85) .. (6.22,1.20)
  .. controls (6.22,1.65) and (6.75,1.92) .. (7.35,1.80)
  .. controls (8.10,1.65) and (8.42,1.15) .. (8.42,0.50)
  .. controls (8.42,-0.40) and (7.70,-1.35) .. (7.00,-1.35)
  -- cycle
}

\def\smallouter{
  (1.25,0.15)
  .. controls (1.25,0.85) and (1.80,1.23) .. (2.50,1.23)
  .. controls (3.20,1.23) and (3.75,0.90) .. (3.75,0.30)
  .. controls (3.75,-0.22) and (3.15,-0.62) .. (2.60,-0.68)
  .. controls (1.90,-0.75636) and (1.25,-0.40) .. (1.25,0.15)
  -- cycle
}

\path[even odd rule,fill=sliceBlue]
  \largeouter \innercurve \smallouter;

\draw[boundary] \largeouter;
\draw[boundary] \innercurve;
\draw[boundary] \smallouter;


\node[anchor=east] at (2.42,2.20) {Outer curves};
\draw[leader] (2.49,2.25)
  .. controls (2.88,2.50) and (3.15,2.53) .. (3.45,2.15);
\draw[leader] (1.63,2.01)
  .. controls (1.60,1.65) and (1.93,1.45) .. (2.50,1.23);

\node[anchor=west] at (9.38,2.38) {Inner curve};
\draw[leader] (9.48,2.16)
  .. controls (9.05,1.88) and (8.65,1.58) .. (8.15875,1.33750);

\draw[vector] (1.362500,-0.222457) -- (0.484564,-0.798415)
  node[left=3pt] {$n$};
\draw[vector] (1.362500,-0.222457) -- (1.938459,-1.100393)
  node[below=3pt] {$T_e$};

\draw[vector] (8.642931,-1.342706) -- (9.437776,-2.028798)
  node[right=3pt] {$n$};
\draw[vector] (8.642931,-1.342706) -- (9.329023,-0.547861)
  node[above right=1pt] {$T_e$};

\draw[vector] (6.807231,-1.311862) -- (7.242948,-0.356535)
  node[above=3pt] {$n$};
\draw[vector] (6.807231,-1.311862) -- (5.851904,-0.876146)
  node[left=3pt,vector label] {$T_e$};


\end{tikzpicture}
\caption{The shaded region is the slice $K_{e,t}$ within the plane $P_{e,t}$. The vector $e$ points out of the page. In this case, $\chi(K_{e,t})=1$.}
\label{Ket_fig}
\end{figure}
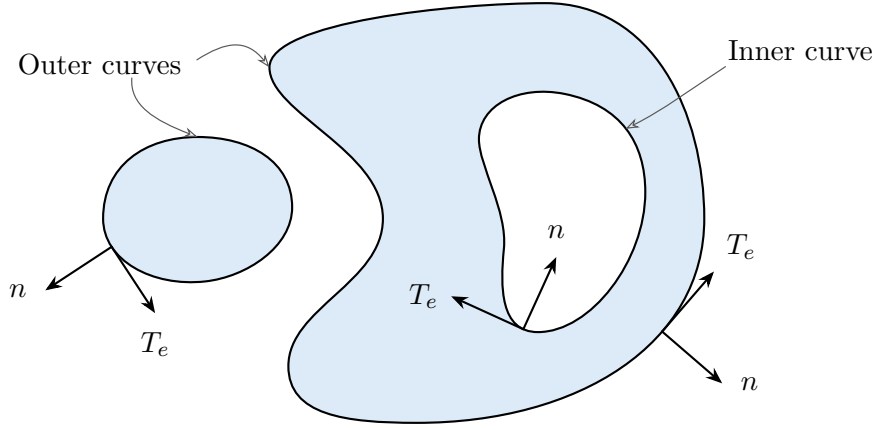

The normal vector $\nu$ is then a linear combination of $e$ and $n$, and the component of $\nu$ in the direction of $n$ is given by
\beql{nu_n}
\langle\nu,n\rangle = |\grad h_e|.
\eeq
If in addition we are at a point on $\partial K$ at which the normal vector $\nu$ is differentiable
(and almost all points on $\partial K$ 
have this property) then we can define the geodesic curvature  $\ka_e$ of $\partial K_{e,t}$ by
$$\grad_{T_e}T_e=-\kappa_e n.$$
Since $\langle \nu, T_e\rangle\equiv 0$, we can then compute
$$0=T_e\langle \nu, T_e\rangle=\langle \grad_{T_e}\nu, T_e\rangle
+\langle \nu, \grad_{T_e}T_e\rangle
=A(T_e,T_e)-\ka_e\langle\nu,n\rangle,$$
and so by \eqref{nu_n} we have
\beql{A_kappa}
A(T_e,T_e)=\ka_e |\grad h_e|.
\eeq
Here, $A(\cdot,\cdot)$ is the second fundamental form of $\partial K$ defined by
$A(v,w)=\langle \grad_v\nu,w\rangle$.
Fix now $e\in S^2$ within the set of full measure for which 
the set of critical points of $h_e$ is finite.
For each regular value $t\in\R$ of $h_e$, 
the Gauss-Bonnet theorem (for example) applied to $K_{e,t}$ gives
$$\int_{\partial K_{e,t}} \ka_e ds = 2\pi \chi(K_{e,t}).$$
The coarea formula, applied with the Lipschitz height function $h_e$, then gives
\beql{pre_S2_int}
2\pi\int_{\R} \chi(K_{e,t}) dt = 
\int_{\R} \int_{\partial K_{e,t}} \ka_e ds dt
= \int_{\partial K} \ka_e |\grad h_e| dvol
=\int_{\partial K}A(T_e,T_e) dvol
\eeq
by \eqref{A_kappa}.
For this application of coarea to be valid, and not be affected by the finitely many critical points, we require the absolute integrability given by
$$\int_{\R} \int_{\partial K_{e,t}} |\ka_e| ds dt
=\int_{\partial K}|A(T_e,T_e)| dvol\leq
\int_{\partial K}|A| dvol<\infty.$$
Finally we would like to integrate out the variable $e\in S^2$.
As $e$ varies over $S^2$, the vector $T_e=\frac{e\times \nu}{|e\times \nu|}$ varies  over the circle $\cc:=\nu^\perp \intersect S^2$. 
If we parametrise this circle by $\th$ and parametrise $S^2\setminus\{\nu,-\nu\}$ by
$(\th,s)$ where $s=\langle e,\nu \rangle \in (-1,1)$, then $T_e$ is independent of $s$.
We can use Archimedes' principle again to write the area element $d\si(e)$ on $S^2$ as $d\th ds$.
At each point on $\partial K$ at which the second fundamental form exists, 
we have
$$\int_{S^2} A(T_e,T_e) d\si(e)=\int_{-1}^1\int_{\cc} A(T_e,T_e) d\th(e)ds
=2\int_{\cc} A(T_e,T_e) d\th(e)=2\int_{\cc} A(e,e) d\th(e),$$
where the final equality follows because $T_e$ is just a $\frac{\pi}2$ rotation of $e\in\cc$ within $\cc$.
If we pick an orthonormal basis $\{v_1,v_2\}$ of the tangent space $\nu^\perp$ 
rotated to represent principal curvature directions, with curvatures $k_1$ and $k_2$ respectively, then we can write $e=\cos\th v_1 + \sin \th v_2$ and 
\beqa
A(e,e) & =A(v_1,v_1) \cos^2\th  + 2A(v_1,v_2)\cos\th\sin\th  + A(v_2,v_2)\sin^2\th \\
&= k_1 \cos^2\th + k_2 \sin^2\th.
\eeqa
Thus
$$\int_{S^2} A(T_e,T_e) d\si(e)=
2\int_{\cc} A(e,e) d\th(e) = 2\int_0^{2\pi} (k_1 \cos^2\th + k_2 \sin^2\th)d\th=2\pi(k_1+k_2)=4\pi H.$$
Applying this to the $S^2$ integral of \eqref{pre_S2_int} gives
$$2\pi\int_{S^2}\int_\R \chi(K_{e,t}) dt\,d\si(e) =
\int_{\partial K}\int_{S^2}A(T_e,T_e)d\si(e) dvol
=4\pi \int_{\partial K} H dvol,$$
which is exactly the identity \eqref{crofton_identity} that we sought.
\end{proof}

To see why the Crofton formula of Lemma \ref{crofton_lem} might be useful in proving the main Theorem \ref{main_thm}, consider the case that $M$ is a smooth closed surface that bounds a convex region $F\subset \R^3$. In that case, for every affine plane $P\subset \R^3$ intersecting the interior of $F$, the slice $F\intersect P$ 
is a closed disc with smooth boundary, so 
$\chi(F\intersect P)=1$.
Consequently for each $e\in S^2$ we have
$$\int_{\R} \chi(F\intersect P_{e,t})dt = \w_M(e),$$
and the Crofton formula given in \eqref{crofton_identity} implies
\beql{croft_conv}
\int_{M} H dvol = \half\int_{S^2}\w_M(e) d\si(e).
\eeq
Keeping in mind the fact that $H\geq 0$ for convex surfaces,
we have improved the first inequality of \eqref{gen_conc},
to an equality, in this special case, recovering 
the identity of Minkowski mentioned in Remark \ref{mink_rmk}.
Thus we have proved the full diameter conjecture in the case of surfaces that bound a convex set. The deduction of the sharp diameter estimate from \eqref{croft_conv} for convex surfaces is already described in the work of Miura \cite{miura_selecta}.

To prove the full result of Theorem \ref{main_thm} without the assumption of convexity we must work harder to overcome the problem that the slices of the interior of $M$ might not typically just consist of a single disc, so the Euler characteristic integrand of the Crofton formula \eqref{crofton_identity} might not be $1$.
The most tempting thing to try is to replace the interior of the boundary $M$ by its convex hull, which does not change the widths or diameter. For this to be useful in proving the original conjecture, we would need to bound the total absolute mean curvature of the boundary of the convex hull by the total absolute mean curvature of $M$. 
For the sharper estimate of Theorem \ref{main_thm} we would need to bound the 
integral of the positive part of the mean curvature of the boundary of the convex hull by the integral of the positive part of the  mean curvature of $M$.
These relationships  turn out to be true, but the only way we know how to prove this is by an application of the main Theorem \ref{main_thm} that we are currently trying to prove! Indeed, if $F$ is the closure of the interior of $M$, and $F^c$ is its convex hull, then \eqref{croft_conv} and Theorem \ref{main_thm} will tell us that
\beqa
\int_{\partial F^c} (H_{\partial F^c})_+ dvol 
&= \int_{\partial F^c} H_{\partial F^c} dvol = 
\half\int_{S^2}\w_{\partial F^c}(e) d\si(e)=
\half\int_{S^2}\w_M(e) d\si(e) \\
&\leq
\int_{M} (H_M)_+ dvol,
\eeqa
and additionally, as a consequence, that 
\beq
\int_{\partial F^c} |H_{\partial F^c}| dvol = 
\int_{\partial F^c} (H_{\partial F^c})_+ dvol 
\leq
\int_{M} (H_M)_+ dvol 
\leq \int_{M} |H_M| dvol.
\eeq
To avoid this circularity in the proof, we take a different type of hull, considered by Spadaro, in place of the convex hull.
In a preprint \cite{spadaro1}, and a later publication \cite{spadaro2},
Spadaro considered the \emph{mean-convex hull}. To define this, 
we first define the class $\s$ of images $\Si$ of smooth minimal embeddings in $\R^3$ of noncompact boundaryless surfaces, with $\Si$ bounded and of finite area.
We denote by $\ca$
the set of \emph{global barriers} $\Th\subset \R^3$, i.e. the set of subsets 
$\Th\subset \R^3$ with the property that 
$$\Si\in\s\quad\text{ with edge }\quad\overline{\Si}\setminus \Si \subset \Th \quad
\implies\quad \Si\subset\Th.$$ 
This is a precise way of saying 
$$\Si\text{ a minimal surface with }\partial \Si\subset\Th\quad
\implies\quad \Si\subset\Th.$$ 
With a slightly larger class $\s$, but one that ultimately gives the same mean-convex hull (see Remark \ref{min_surf_class_rmk}) Spadaro defined:
\begin{defn}[\cite{spadaro1}]
Given a 
subset $\Om\subset \R^3$, the \emph{mean-convex hull} of $\Om$ is defined to be 
$$\Om^{mc}:=\bigcap_{\Om\subset\Th\in \ca}\Th.$$
\end{defn}
By the convex hull property of minimal surfaces, the convex hull of $\Om$ will always serve as a global barrier. In general, the convex hull can be strictly larger than the mean-convex hull. 
See Spadaro \cite{spadaro1, spadaro2} for a discussion of the relationships between mean-convex hulls, mean-convex sets, and outward minimising sets.
%
%
\begin{thm}[Special case of Spadaro \cite{spadaro1, spadaro2}]
\label{spadaro_thm}
Suppose  $F\subset \R^3$ is the closure of an open bounded set with $C^{1,1}$ boundary.
Then the mean-convex hull $F^{mc}$ is a closed, bounded, mean-convex set with 
$C^{1,1}$ boundary. Moreover, any new boundary set $\partial (F^{mc})\setminus \partial F$ is a smooth minimal surface.
\end{thm}
The idea of Spadaro to construct the mean-convex hull was to \emph{shrink-wrap}
the original set $F$. Slightly more precisely, starting with an enclosing surface, such as a large sphere, Spadaro iteratively reduces its area as quickly as possible while keeping $F$ enclosed. This gives a discrete approximation to mean curvature flow with an obstacle. He then takes the large-time limit, followed by the limit as the time step goes to zero.

\begin{rmk}
\label{min_surf_class_rmk}
As mentioned earlier, Spadaro took a larger class $\s$ of minimal surfaces than the one we considered above. In theory, it might then have resulted in a bigger mean-convex hull $\widehat{F^{mc}}$ than the one we are considering. However, the new part of the boundary $\partial (\widehat{F^{mc}})\setminus \partial F$ is a finite area minimal surface,
and lies in our class $\s$, with its edge within $F$.
Therefore the entire boundary $\partial (\widehat{F^{mc}})$ 
lies within any global barrier $\Th$, with respect to our $\s$, 
that contains $F$.
But then, given any point $p$ in the interior of $\widehat{F^{mc}}$, we can pick any plane $P$ through $p$ and consider the minimal surface $\Si$ consisting of the connected component of $P\intersect \Interior(\widehat{F^{mc}})$
containing $p$, which also lies within our class $\s$ and has its
edge within $\partial (\widehat{F^{mc}})\subset \Th$, and conclude from the global barrier property that $p\in\Si\subset \Th$.
Since $p\in P$ was an arbitrary point in the interior of $\widehat{F^{mc}}$,
and $\partial (\widehat{F^{mc}})\subset \Th$,
we have $\widehat{F^{mc}}\subset \Th$.
We deduce that the mean-convex hull $\widehat{F^{mc}}$ of Spadaro must lie within the mean-convex hull $F^{mc}$ we consider here, so they must coincide.
Spadaro himself uses an argument of this form in his work \cite[\S 5.4, page 373]{spadaro2}.
\end{rmk}


All we will need of Theorem \ref{spadaro_thm} is the following
immediate corollary that happens to be 
a little easier to prove directly. The set $K$ below can be taken to be the mean-convex hull of $F$.
\begin{cor}
\label{mean_convex_cor}
Suppose  $F\subset \R^3$ is the closure of an open bounded set with $C^{1,1}$ boundary.
Then there exists a set $K\subset\R^3$ that is also 
the closure of an open bounded set with $C^{1,1}$ boundary,
with the following properties.
\begin{compactenum}
\item
$F\subset K \subset F^c$, where $F^c$ is the convex hull of $F$;
\item
$\partial K \setminus \partial F$ is a smooth minimal surface;
\item
\label{property_4}
If $P$ is an affine plane in $\R^3$ and $\Ga$ is a smooth closed embedded curve in $P$ lying in the interior of $K$, then the closed disc in $P$ bounded by $\Ga$ lies wholly within $K$.
\end{compactenum}
\end{cor}


A major advantage of the mean-convex hull over the convex hull is that we can relate the total mean curvature of a closed bounded subset with $C^{1,1}$ boundary to the total mean curvature of its mean-convex hull. More precisely we have the following.
\begin{lem}
\label{H_lem}
Suppose $F\subset K$ are the closures of open bounded subsets of $\R^3$ with $C^{1,1}$ boundaries, with $\partial K \setminus \partial F$ a smooth minimal surface. 
Define the mean curvatures $H_{\partial F}$ and $H_{\partial K}$ with respect to outward normal vectors. Then 
\beql{mean_curv_comp}
\int_{\partial K} H_{\partial K} dvol = 
\int_{\partial K \intersect \partial F} H_{\partial F}dvol 
\leq \int_{\partial F}(H_{\partial F})_+ dvol.
\eeq
\end{lem}
\begin{proof}
The inequality of \eqref{mean_curv_comp} is straightforward because
$$\int_{\partial K \intersect \partial F} H_{\partial F}dvol
\leq \int_{\partial K \intersect \partial F} (H_{\partial F})_+ dvol
\leq\int_{\partial F} (H_{\partial F})_+ dvol.$$
We are therefore reduced to proving the equality of \eqref{mean_curv_comp}.
Because $H_{\partial K}=0$ on $\partial K \setminus \partial F$ by hypothesis, 
this reduces to proving 
\beql{intersection_reduction}
\int_{\partial K \intersect \partial F} H_{\partial K} dvol = 
\int_{\partial K \intersect \partial F} H_{\partial F}dvol.
\eeq
Given an arbitrary point $p\in \partial K \intersect \partial F$, i.e. an arbitrary contact point, the tangent planes to $\partial K$ and $\partial F$
must coincide because of the inclusion $F\subset K$ and the regularity of the boundaries.
In a neighbourhood of $p$ we can write  $\partial K$ and $\partial F$ as $C^{1,1}$ graphs over this plane, of functions $k$ and $f$ respectively.
We may assume that near $p$ the sets $K$ and $F$ lie below these graphs rather than above. The inclusion $F\subset K$ then implies that $f\leq k$ locally,
so we can define a $C^{1,1}$ function $g:=k-f$ near to $p$ with the property 
that $g\geq 0$. 
Locally, the contact set $\partial K \intersect \partial F$ is described by the zero set of $g$. 
On that zero set the derivative of $g$ vanishes  everywhere, since $g\geq 0$ locally. 

Meanwhile, each component of the derivative of $g$ is Lipschitz, and every Lipschitz function has the property that its derivative exists, and vanishes, 
almost everywhere on its zero set. Indeed, it will vanish at every point for which both the Lipschitz function is differentiable and the zero set 
is of density one. Consequently, we find that locally the Hessian of $g$ must both exist and vanish at almost every point of the contact set $\{g=0\}$. Thus, locally, $f$ agrees with $k$, up to second order derivatives at almost every point in $\partial K \intersect \partial F$.
Because  we can write the mean curvatures of $\partial K$ and $\partial F$ in terms of their local defining functions
$k$ and $f$ and their first and second derivatives, we find that locally to $p$, we have 
$H_{\partial K}=H_{\partial F}$ at almost all points of 
$\partial K \intersect \partial F$.
But $p$ was an arbitrary point in the contact set, so 
\eqref{intersection_reduction} follows.
\end{proof}

Property \ref{property_4} of Corollary \ref{mean_convex_cor} will allow us to ensure that the Euler characteristic of slices of $K$ must be $1$ or more.
\begin{lem}
\label{no_inner_curves}
Suppose $K\subset \R^3$ is the closure of an open bounded set with $C^{1,1}$ boundary, 
satisfying Property \ref{property_4} of Corollary \ref{mean_convex_cor}.
Then for every $e\in S^2$, and every regular value $t\in \R$ of 
$h_e:\partial K\to\R$  defined by
$h_e(x)=\langle x,e\rangle$,
every boundary curve of $\partial K \intersect P_{e,t}$ is an outer curve, and hence 
$$\chi(K_{e,t})\geq 1$$
if $K_{e,t}$ is non-empty.
\end{lem}
As we have used before, the facts that the slice $K_{e,t}:=K \intersect P_{e,t}$ is at a regular value of $h_e$, and that $\partial K$ is $C^{1,1}$ and compact, imply that either the slice is empty, or it consists of a  closed bounded set whose boundary is a finite collection of pairwise disjoint $C^{1,1}$ curves.
\begin{proof}[Proof of Lemma \ref{no_inner_curves}]
It is straightforward to approximate our finitely many boundary components of 
$K_{e,t}$ by nearby smooth curves that lie within $K_{e,t}$. For example, one could solve the PDE $-\lap u=1$ on $K_{e,t}$ with zero boundary data on $\partial K_{e,t}$, and then consider level sets $\{u=\ep\}$ for small $\ep>0$.

Pick any one of these boundary components and its corresponding smooth curve.
By Property \ref{property_4} of Corollary \ref{mean_convex_cor}, the interior of this curve in $P_{e,t}$ will already lie in $K_{e,t}$, so adding in this disc would not change the boundary components of $K_{e,t}$. 
Therefore, by the definition of `inner' curve from Remark \ref{IFT_rmk}, this arbitrary boundary curve, and therefore every boundary curve, must be an `outer' curve. 
The Euler characteristic $\chi(K_{e,t})$ is the number of outer curves.
\end{proof}

\begin{proof}[Proof of the main Theorem \ref{main_thm} and the extension of Remark \ref{C11regthm}]
For the closed $C^{1,1}$ surface $M$ embedded in $\R^3$, we denote by 
$F$ the closure of the open bounded subset of $\R^3$ having $M$ as its boundary.
By Corollary \ref{mean_convex_cor}, there exists a closed bounded set $K\subset\R^3$
satisfying the four conclusions of that corollary.
By Lemma \ref{H_lem}, we have
\beql{mean_curv_comp2}
\int_{\partial K} H_{\partial K} dvol 
\leq \int_M(H_{M})_+ dvol.
\eeq
For $e\in S^2$, consider the function $x\mapsto\langle x,e\rangle$ that we have  called $h_e$ when restricted to $\partial K$. Its image over $M$, which we  abbreviate $\langle M,e\rangle$, is a closed interval $\ci_e\subset \R$ of length $\w_M(e)$ by the connectedness of $M$. By the definition of $F$, the image $\langle F,e\rangle$ also equals $\ci_e$.
By the inclusions $F\subset K \subset F^c$ of Corollary \ref{mean_convex_cor}, 
this gives
\beql{Ke_image}
\langle K,e\rangle = \ci_e. 
\eeq
Next, we invoke the Crofton formula of Lemma \ref{crofton_lem} to give
\beql{crofton_identity2}
\int_{\partial K} H_{\partial K} dvol = \half\int_{S^2}\int_{\R} \chi(K_{e,t})dt d\si(e).
\eeq
By Lemma \ref{no_inner_curves} and \eqref{Ke_image}, 
for every $e\in S^2$ and every regular value $t\in\ci_e$,
we have $\chi(K_{e,t})\geq 1$
and hence
\beql{width_comp2}
\int_\R \chi(K_{e,t})dt=\int_{\ci_e} \chi(K_{e,t})dt\geq |\ci_e|=\w_M(e).
\eeq
In fact, it will suffice to prove this for almost all $e\in S^2$, so we can appeal to Remark \ref{IFT_rmk} and assume that there are finitely many critical values in $\ci_e$,
relieving the need for Sard's theorem for $C^{1,1}$ functions to justify \eqref{width_comp2}.
Combining \eqref{mean_curv_comp2}, \eqref{crofton_identity2}
and \eqref{width_comp2} gives 
\beql{crofton_identity3}
\int_M(H_{M})_+ dvol\geq 
\int_{\partial K} H_{\partial K} dvol 
= \half\int_{S^2}\int_{\R} \chi(K_{e,t})dt d\si(e)
\geq \half\int_{S^2} \w_{M}(e) d\si(e),
\eeq
which is the first inequality of \eqref{gen_conc}.
Together with Remark \ref{proof_reduction}, this concludes the proof.
\end{proof}

Finally, we give a proof of our estimate for surfaces with boundary, by 
combining the method of Miura \cite{miura_sns} with our new diameter estimates.

\begin{proof}[Proof of Theorem \ref{bdry_thm}]
By compactness of $M$ and its boundary, there exists $\ep_0>0$ so that 
both $M$ and $\partial M$ admit $\ep_0$-tubular neighbourhoods.
For $\ep\in (0,\ep_0)$, we consider all points in $\R^3$ that are exactly a distance $\ep$ from $M$ via a minimising geodesic that meets 
$M$ orthogonally; we call this surface $M_\ep$. 
Locally, for $\ep$ very small the surface $M_\ep$ looks very much like two copies of $M$, displaced in opposite normal directions by a distance $\ep$.
Globally, $M_\ep$ will be  a two-sheeted surface if $M$ is two-sided. Otherwise (e.g. for a M\"obius band) $M_\ep$ will be a single connected surface. Either way, after defining a unit normal on $M_\ep$ that is pointing away from $M$, in the limit $\ep\downto 0$, intuitively, $M_\ep$ collapses to two copies of $M$ with different normals, so with mean curvature differing by a sign. 
More precisely (cf. Miura \cite{miura_sns}) we have
$$\lim_{\ep\downto 0}
\int_{M_\ep}(H_{M_\ep})_+ dvol =\int_M |H| dvol.$$
Meanwhile, $M_\ep$ is part of a larger closed embedded surface $\widetilde M_\ep$, of regularity $C^{1,1}$, consisting of all points in $\R^3$ exactly a distance $\ep$ from $M$, without the restriction that the minimising geodesic to $M$ hits $M$ orthogonally. This serves to connect the sheets of $M_\ep$ near their boundary by adding half-tubes around $\partial M$ of cross-sectional radius $\ep$, resulting in a closed embedded surface. 
One can consult Miura \cite[Figure 2.1]{miura_sns} for an example illustration.
For small $\ep$, these half-tubes have area of order $\pi \ep \Length (\partial M)$, and positive mean curvature approximately $1/(2\ep)$. 
More precisely (cf. Miura \cite{miura_sns}) we have
$$\lim_{\ep\downto 0} \int_{\widetilde M_\ep\setminus M_\ep}(H_{\widetilde M_\ep})_+ dvol = \frac{\pi}{2} \Length (\partial M),$$
and thus
\beq
\label{ep_to_0_lim1}
\lim_{\ep\downto 0} \int_{\widetilde M_\ep} (H_{\widetilde M_\ep})_+ dvol
= \int_M |H| dvol+
\frac{\pi}{2} \Length (\partial M).
\eeq
We also observe that $\widetilde M_\ep$ is $\ep$-close to $M$ in the Hausdorff topology, so 
\beq
\label{ep_to_0_lim2}
\w_{\widetilde M_\ep}(e) \to \w_M(e) \text{ uniformly in }e\text{ as }\ep\downto 0.
\eeq
Given that $\widetilde M_\ep$ is a closed embedded surface of $C^{1,1}$ regularity, we can apply the main Theorem \ref{main_thm} to find that 
$$\int_{\widetilde M_\ep} (H_{\widetilde M_\ep})_+ dvol\geq 
\half \int_{S^2}\w_{\widetilde M_\ep}(e) d\si(e).$$
Taking the limit $\ep\downto 0$, using \eqref{ep_to_0_lim1}
and \eqref{ep_to_0_lim2} gives 
$$\int_M |H| dvol+\frac{\pi}{2} \Length (\partial M)
\geq 
\half \int_{S^2}\w_{M}(e) d\si(e),$$
which is the first inequality of \eqref{bdry_est}.
The second inequality follows from Remark \ref{proof_reduction} as before.
\end{proof}


\emph{Acknowledgements:} 
Thanks to Jos\'e Escobar, Tatsuya Miura and Christian Scharrer for interesting conversations about this problem over the years.
For the purpose of open access, the author has applied a creative commons attribution (CC BY) licence to any author accepted manuscript version arising.

\emph{Tool use statement:} 
The problem considered is due to the author. The proof was developed in collaboration with, and as a test of, commercially available models of ChatGPT as of September/October 2026, in particular 5.6 Sol and 6 Astra. This included rapid assessments of different slicing methods, potential methods, hull methods, symmetrisation methods and the generation of new ideas. It also helped in assembling a representative sample of the applications of our earlier work and in navigating the classical integral geometry literature. This paper was assembled from scratch by the author based on these interactions. Only the bibliography has any text or formulae directly pasted from AI generated output. AI was used to convert Figure \ref{Ket_fig} from handwritten form to tikz form. AI was used to pick up minor errors in a draft version.
For the avoidance of doubt, the author takes full responsibility for the correctness of the paper.


\noindent

\noindent
\url{http://warwick.ac.uk/fac/sci/maths/people/staff/peter\_topping/}\\
\noindent
{\sc Mathematics Institute, University of Warwick, Coventry,
CV4 7AL, UK.}

\end{document}